\documentclass[12pt]{article}
\usepackage{amsthm,amssymb,hyperref,palatino}
\usepackage{tikz}
\usepackage{fullpage}

\newtheorem{thm}{Theorem}

\newtheorem{lem}{Lemma}

\newtheorem{rem}{Remark}
\newtheorem{dfn}{Definition}

\newcommand{\<}{\langle}
\renewcommand{\>}{\rangle}

\begin{document}
\title{More Tales of Hoffman: bounds for the vector chromatic number of a graph}

\author{Pawel Wocjan\thanks{\texttt{wocjan@cs.ucf.edu}, Department of Computer Science, University of Central Florida, USA}  \quad Clive Elphick\thanks{\texttt{clive.elphick@gmail.com}, School of Mathematics, University of Birmingham, Birmingham, UK} \quad David Anekstein\thanks{\texttt{aneksteind@gmail.com}}}

\maketitle

\abstract{Let $\chi(G)$ denote the chromatic  number of a graph and $\chi_v(G)$ denote the vector chromatic number. For all graphs $\chi_v(G) \le \chi(G)$ and for some graphs $\chi_v(G) \ll \chi(G)$. Galtman proved that Hoffman's well-known lower bound for $\chi(G)$ is in fact a lower bound for $\chi_v(G)$. We prove that two more spectral lower bounds for $\chi(G)$ are also lower bounds for $\chi_v(G)$. We then use one of these bounds to derive  a new characterization of $\chi_v(G)$.}

\section{Introduction}\label{sec:1}

For any graph $G$ let $V$ denote the set of vertices where $|V| = n$, $E$ denote the set of edges where $|E| = m$, $A$ denote the adjacency matrix, $\chi(G)$ denote the chromatic number and $\omega(G)$ the clique number. Let $\mu_1 \ge \mu_2 \ge ... \ge \mu_n$ denote the eigenvalues of $A$ and let $s^+$ and $s^-$ denote the sum of the squares of the positive and negative eigenvalues of $A$, respectively. Let $\overline{G}$ denote the complement of $G$. 

Let $D$ be the diagonal matrix of vertex degrees, and let $L = D - A$ denote the Laplacian of $G$ and $Q = D  + A$ denote the signless Laplacian of $G$. The eigenvalues of $L$ are $\lambda_1 \ge \ldots \ge \lambda_n = 0$ and the eigenvalues of $Q$ are $\delta_1 \ge \ldots \ge \delta_n$.

\section{Vector chromatic numbers and theta functions}\label{sec:two}

In 1979 Lov\'asz \cite{lovasz79} defined the theta function, $\vartheta(G)$,  that is now named after him, in order to upper bound the Shannon capacity, $c(G)$, of a graph, and proved that $c(C_5) = \vartheta(C_5) = \sqrt{5}$. He also proved that $\omega(G) \le \vartheta(\overline{G}) \le \chi(G)$. Schrijver and Szegedy subsequently defined variants of the Lov\'asz theta function, which are denoted $\vartheta'(G)$ and $\vartheta^+(G)$ respectively, where $\vartheta'(G) \le \vartheta(G) \le \vartheta^+(G)$. All three theta functions can be approximated to within a fixed $\epsilon$ in polynomial time using semidefinite programming (SDP), even though computing $\omega(G)$ and $\chi(G)$ is NP-hard.

In parallel with the use of these theta functions, various vector chromatic numbers were defined. In 1998 Karger \emph{et al} \cite{karger98} defined the vector chromatic number, $\chi_v(G)$, and the strict vector chromatic number, $\chi_{sv}(G)$, where $\chi_v(G) \le \chi_{sv}(G) \le \chi(G)$. There exist graphs for which $\chi_v(G) \ll \chi(G)$ \cite{feige04}. Karger \emph{et al} \cite{karger98}  also proved that $\chi_{sv}(G) = \vartheta(\overline{G})$,  and Godsil \emph{et al} \cite{godsil17} noted that $\chi_v(G) = \vartheta'(\overline{G}).$ Finally there is what is called the rigid vector chromatic number, $\chi_{rv}(G)$, and Roberson proved (see Section 6.7 of \cite{roberson13}) that $\chi_{rv}(G) = \vartheta^+(\overline{G})$. So to summarise

\begin{equation}
\omega(G) \le \chi_v(G) = \vartheta'(\overline{G}) \le \chi_{sv}(G) = \vartheta(\overline{G}) \le \chi_{rv}(G) = \vartheta^+(\overline{G}) \le \chi(G).
\end{equation}

In this paper we  focus on lower bounds for $\chi_v(G)$ so it is only necessary to include the following definition.

\begin{dfn}[Vector chromatic number $\chi_v(G)$]\label{def:chiv}
Given a graph $G = (V, E)$ on $n$ vertices, and a real number $k \ge 2$, a vector $k$-coloring of $G$ is an assignment of unit vectors $u_i\in\mathbb{R}^n$ to each vertex $i \in V$, such that for any two adjacent vertices $i$ and $j$ 
\begin{equation}\label{eq:inner}
\langle u_i , u_j \rangle \le  -\frac{1}{k - 1}.
\end{equation}
The vector chromatic number $\chi_v(G)$ is the smallest real number $k$ for which a vector $k$-coloring exists. The vector $k$-coloring can always be assumed to be in dimension $n$. 
\end{dfn}

\section{Spectral lower bounds for chromatic numbers}\label{sec:three}

Most of the known spectral lower bounds for the chromatic number can be summarised as follows:

\begin{equation}\label{bounds}
1 + \max\left(\frac{\mu_1}{|\mu_n|} , \frac{2m}{2m - n\delta_n} , \frac{\mu_1}{\mu_1 - \delta_1 + \lambda_1} ,   \frac{s^\pm}{s^\mp}\right) \le \chi(G) , 
\end{equation}
where, reading from left to right, these bounds are due to Hoffman \cite{hoffman70}, Lima \emph{et al} \cite{lima11}, Kolotilina \cite{kolotilina11}, and Ando and Lin \cite{ando15}.  It should be noted that Nikiforov \cite{nikiforov07} pioneered the use of non-adjacency matrix eigenvalues to bound $\chi(G)$.

Note that for regular graphs the first three bounds are equal. Some of these bounds are further generalised in Elphick and Wocjan \cite{elphick15}, which for reasons discussed in Section {\ref{sec:five}} we exclude here. Several of these bounds equal two for all bipartite graphs.

Wocjan and Elphick \cite{wocjan18} strengthened (\ref{bounds}) by proving that the Ando and Lin bound is a lower bound for the quantum chromatic number, $\chi_q(G)$, with arbitrary Hermitian weight matrices. Wocjan and Elphick \cite{wocjan182} further strengthened (\ref{bounds}) by proving that the Kolotilina and Lima \emph{et al} bounds are lower bounds for the vectorial chromatic number, $\chi_{vect}(G) = \lceil \vartheta^+(\overline{G}) \rceil$, again with arbitrary Hermitian weight matrices.  

Galtman \cite{galtman00} provides eight characterizations of $\chi_v(G)$. The fifth of these is that:
\begin{equation}\label{eq:galtman}
\chi_v(G) = 1 + \max_W \left(\frac{\mu_1(W)}{|\mu_n(W)|}\right),
\end{equation}
where $W$ is an arbitrary non-negative\footnote{Non-negative means that all matrix entries are non-negative.} weight matrix. This shows that the Hoffman bound is a lower bound for the vector chromatic number, $\chi_v(G) = \vartheta'(\overline{G})$, but for non-negative weight matrices only. This bound was also independently obtained by Bilu \cite{bilu06}.

We prove below that the bounds due to Lima \emph{et al} and Kolotilina are also lower bounds for $\chi_v(G)$. It is straightforward to amend our proofs to show that the Lima \emph{et al} and the Kolotilina bounds remain lower bounds for $\chi_v(G)$ with arbitrary non-negative weight matrices. In the case of the Lima \emph{et al} bound this involves replacing $2m$ in the numerator with the sum of the off-diagonal entries of the weight matrix and $2m$ in the denominator with the trace of the weight matrix. In Section {\ref{sec:four}} we use the Lima \emph{et al} bound to prove a new characterization of the vector chromatic number. As discussed by Galtman \cite{galtman00}, removing the non-negativity constraint would provide lower bounds for $\chi_{sv}(G)$, which can exceed $\chi_v(G)$.

\section{Proof for the Lima bound}\label{sec:four}

\begin{thm}\label{thm:lima}
For any graph $G$
\begin{equation}
1 + \frac{2m}{2m- n\delta_n} \le \chi_v(G).
\end{equation}
\end{thm}

\begin{proof}
Let $u_1,\ldots,u_n\in\mathbb{R}^n$ be the unit vectors on which the vector chromatic number $\chi_v$ is attained. That is $\langle u_i, u_j \rangle \le -1/(\chi_v - 1)$ for all $ij \in E$. 

Let $e_1,\ldots,e_n$ denote the standard basis of $\mathbb{R}^n$. Define the vector
\begin{equation}
v = \sum_{i=1}^n e_i \otimes u_i \in \mathbb{R}^n \otimes \mathbb{R}^n\,.
\end{equation}
Let $q_{ij}$ denote the entries of the signless Laplacian $Q$. We have
\begin{eqnarray}
n \cdot \delta_n 
& = &
\<v, v\>  \cdot \delta_n \\
& \le &
\< v, (Q\otimes I_n) v \> \\
& = &
\sum_{i,j=1}^n q_{ij} \cdot \< u_i, u_j \> \\
& = &
\sum_{i=1}^n d_i  + 2 \sum_{ij\in E}^n \< u_i, u_j \> \\
& \le &
2m - 2m \cdot \frac{1}{\chi_v - 1}\,.
\end{eqnarray}
We use the Rayleigh principle $\delta_n\le \<v, (Q\otimes I_n) v\> / \<v, v\>$. We then use 
$Q=D+A$, that is, $q_{ii}=d_i$, $q_{ij}=1$ for all $ij\in E$ and $q_{ij}=0$ for all $ij\not\in E$ and $i\neq j$. We finally use
$\< u_i, u_j \> \le - 1/ (\chi_v - 1)$ for all $ij\in E$.
\end{proof}

%%% alternative proof

We also present an alternative proof of Theorem~\ref{def:chiv}. This proof does not make use of the definition of the vector chromatic number in terms of certain vectors as in Definition~\ref{def:chiv}.  Instead, we rely on the third characterization of $\chi_v(G)$ in Section 3 of \cite{galtman00} which is as follows.

\begin{equation}\label{eq:galtmanHoffman}
    \chi_v(G) = \max_B \sum_{i,j = 1}^n b_{ij},
\end{equation}
where $B = (b_{ij})$ is a non-negative symmetric positive semi-definite matrix such that $\mathrm{tr}(B) = 1$ and $b_{ij} = 0$ if $i$ and $j$ are distinct non-adjacent vertices. We can now reformulate the above characterization of $\chi_v(G)$ so that the Lima \emph{et al} bound arises as a special case.

%%%

\begin{thm}
For any graph $G$
\begin{equation}
\chi_v(G) = 1 + \max_W \left(\frac{\sum_{i\neq j} w_{ij}}{\mathrm{tr}(W) - n\lambda_{\min}(W)}\right),    
\end{equation}
where $W = (w_{ij})$ is a non-negative weight matrix and $\lambda_{\min}(W)$ denotes the minimum eigenvalue of $W$.
\end{thm}

\begin{proof}
Let $W$ be an arbitrary non-negative symmetric matrix.  Then, the matrix 
\begin{equation}
    B=\frac{W - \lambda_{\min}(W) I}{\mathrm{tr}(W) - n \lambda_{\min}(W)}
\end{equation}
is positive semidefinite and $\mathrm{tr}(B) = 1$.  Substituting $B$ into (\ref{eq:galtmanHoffman}) yields the characterization. Setting $W$ equal to the signless Laplacian $Q$ yields the Lima \emph{et al} bound as a special case.

\end{proof}

%%%

\section{Proof for the Kolotilina bound}\label{sec:five}

We briefly recall some standard concepts and results that are needed to prove that the Kolotilina bound is a lower bound for the vector chromatic number.  
Let $X,Y\in\mathbb{C}^{n\times n}$ be two arbitrary Hermitian matrices with eigenvalues $\alpha_1 \ge \alpha_2 \ge \ldots \ge \alpha_n$ and $\beta_1 \ge \beta_2 \ge \ldots \ge \beta_n$, respectively. We say that
$X$ majorizes $Y$, denoted by $X\succeq Y$, if
\begin{equation}
\sum_{i=1}^\ell \alpha_i \ge \sum_{i=1}^\ell \beta_i 
\end{equation}
for all $\ell\in\{1,\ldots,n-1\}$ and
\begin{equation}
\sum_{i=1}^n \alpha_i = \sum_{i=1}^n \beta_i \,.
\end{equation}
Recall that the Schur product of two matrices $M,N\in\mathbb{C}^{n\times n}$, denoted by $M\circ N$ is defined to be the matrix whose entries are the products of the corresponding entries of $M$ and $N$. We say that a Hermitian matrix $M\in\mathbb{C}^{n\times n}$ is positive semidefinite if all its eigenvalues are non-negative. 

The Schur product $M\circ N$ of any two positive semidefinite matrices $M$ and $N$ is positive semidefinite.  Let $u_1,\ldots,u_n\in\mathbb{C}^n$ be a collection of $n$ arbitrary unit vectors.  Their Gram matrix $\Phi=(\Phi_{ij})\in\mathbb{C}^{n\times n}$, whose entries $\Phi_{ij}$ are the inner products $\<u_i,u_j\>$, is positive semidefinite.

We say that a matrix $M\in\mathbb{R}^{n\times n}$ is non-negative if all its entries are non-negative.  Similarly, we say a vector $v\in\mathbb{R}^n$ is non-negative if all its entries are non-negative. Note that $M v$ is non-negative whenever $M$ and $v$ are non-negative. For two matrices $M,N\in\mathbb{R}^{n\times n}$, we write $M\ge N$ to indicate that $M-N$ is non-negative.
For any two non-negative matrices $M,N\in\mathbb{R}^{n\times n}$ and non-negative vector $v\in\mathbb{R}^n$, $M\ge N$ implies $Mv \ge Nv$.

Let $M\in\mathbb{R}^{n\times n}$ be an arbitrary symmetric, non-negative, and irreducible matrix. Then, the eigenvector corresponding to the largest eigenvalue can be chosen to have positive entries. This follows from the proof of the Perron-Frobenius theorem for non-negative irreducible matrices \cite[Chapter 8]{meyer}.

% correlation matrix

For our purposes, it is useful to reformulate the defining condition of a $k$-vector coloring as follows.

\begin{rem}\label{rem}
Note that condition (\ref{eq:inner}) in the definition of a $k$-vector coloring of $G=(V,E)$ can be equivalently formulated as
\begin{equation}\label{eq:equivalent}
\Phi \circ (D - A) \ge D + \frac{1}{k-1} A\,,
\end{equation}
where $D$ denotes the diagonal matrix of vertex degrees $d_i$, $A=(a_{ij})$ the adjacency matrix, and $\Phi=(\Phi_{ij})$ the Gram matrix of the $n$ unit vectors $u_i\in\mathbb{R}^n$ of the $k$-vector coloring, that is, $\Phi_{ij}=\<u_i,u_j\> \le - 1 / (k-1)$ for all  $ij\in E$.
\end{rem}

This reformulation enables us to leverage the well-known Perron-Frobenius theorem because the entries of the matrix on the right hand side of (\ref{eq:equivalent}) are all non-negative.

In the book \cite{zhan02}, correlation matrices are positive semidefinite matrices with ones along the diagonal. It is easy to see that a Gram matrix (of unit vectors) is a correlation matrix and vice versa.\footnote{Every positive semidefinite matrix $\Phi$ can be written as $\Phi=B^* B$ for some $B$ \cite[Exercise I.2.2]{bhatia}. When $\Phi$ is a correlation matrix, then the columns of $B$ of this decomposition are the desired unit vectors whose pairwise inner products form $\Phi$. The other direction is obvious because a Gram matrix is positive semidefinite and its diagonal entries are all one when the corresponding vectors are unit vectors.} Besides the Perron-Frobenius theorem, the result in \cite[Corollary 2.15]{zhan02} plays a central role in showing that the spectral bounds are also lower bounds on the vector chromatic number. We decided to include a proof of this key result. 

\begin{lem}\label{lem:GX}
Let $\Phi\in\mathbb{C}^{n\times n}$ be an arbitrary correlation matrix. Then, for any Hermitian matrix $X\in\mathbb{C}^{n\times n}$ 
\begin{equation}
X \succeq \Phi \circ X\,.
\end{equation}
\end{lem}
\begin{proof}
Set $Y=\Phi\circ X$. Let
\begin{equation}
X = \sum_{j=1}^n \alpha_j P_j\,,\quad\mbox{and}\quad Y=\sum_{i=1}^n \beta_i Q_i
\end{equation}
denote the spectral decompositions of $X$ and $Y$, respectively.  We assume that the eigenvalues of $X$ and $Y$ in the above decompositions are ordered in non-increasing order, that is, 
$\alpha_1 \ge \alpha_2 \ge \ldots \ge \alpha_n$ and $\beta_1 \ge \beta_2 \ge \ldots \ge \beta_n$, respectively. We also assume that the orthogonal projectors $P_j$ and $Q_i$ are one-dimensional. Note that $\sum_{j=1}^n P_j = \sum_{i=1}^n Q_i = I$.

For an arbitrary $i\in\{1,\ldots,n\}$, we can use the spectral decompositions to write $\beta_i$ as follows
\begin{equation}
\beta_i = \sum_{j=1}^n \mathrm{Tr}(Q_i (\Phi \circ P_j)) \, \alpha_j\,.
\end{equation}
For $i, j\in\{1,\ldots,n\}$, define the values $p_{ij} = \mathrm{Tr}(Q_i (\Phi \circ P_j))$ so that $\beta_i = \sum_{j=1}^n p_{ij} \alpha_j$.
Note that equivalently $p_{ij} = v_i^\dagger (\Phi \circ P_j) v_i\ge 0$, where $v_i\in\mathbb{C}^n$ with $Q_i = v_i v_i^\dagger$. Therefore, these values are non-negative because the Schur product $\Phi\circ P_j$ of the two positive semidefinite matrices  $\Phi$ and $P_j$ is positive semidefinite.

We now show that the matrix $P=(p_{ij})$ is doubly stochastic, that is, all row and column sums are equal to $1$.
We have $\mathrm{Tr}(\Phi \circ M)=\mathrm{Tr}(M)$ for all matrices $M\in\mathbb{C}^{n\times n}$ and $\Phi\circ I=I$ because $\Phi$ has ones along the diagonal. These two simple observations and the properties of spectral decompositions imply that
\begin{equation}
\sum_{i=1}^n p_{ij} = \mathrm{Tr}(\Phi\circ P_j) = \mathrm{Tr}(P_j) = 1\,
\end{equation}
and
\begin{equation}
\sum_{j=1}^n p_{ij} = \mathrm{Tr}(Q_i (\Phi \circ I)) = \mathrm{Tr}(Q_i) = 1\,.
\end{equation}
Hence $(\beta_1,\ldots,\beta_n)^T = P (\alpha_1,\ldots,\alpha_n)^T$ for some doubly stochastic matrix $P$.
The Hardy-Littlewood-P{\'o}lya theorem now implies that the spectrum of $X$ majorizes the spectrum of $\Phi\circ X$, that is, 
$(\alpha_1,\ldots,\alpha_n) \succeq (\beta_1,\ldots,\beta_n)$.
\end{proof}

Using the above results, we establish the following theorem.

\begin{thm}\label{thm:vector coloring}
Assume that $A=(a_{ij})\in\mathbb{R}^{n\times n}$ is irreducible and that there exists a correlation matrix $\Phi\in\mathbb{R}^{n\times n}$  such that 
\begin{equation}
\Phi \circ (D - A) \ge D + \frac{1}{k-1} A,
\end{equation}
which is the condition (\ref{eq:equivalent}) in Remark~\ref{rem}. Then, we have
\begin{equation}
\lambda_{\max}(D - A) \ge \lambda_{\max} \left(D+\frac{1}{k-1} A \right)\,.
\end{equation}
\end{thm}

\begin{proof}
We have the following facts:
\begin{eqnarray}
D - A              &\succeq & \Phi \circ (D - A) \label{eq:1} \\
\Phi \circ (D - A) &\ge & D + \frac{1}{k-1} A \label{eq:2}
\end{eqnarray}
Observe that the matrix $D+1/(k-1)A$ is symmetric, non-negative, and irreducible because $A$ has these properties and $D$ is a diagonal matrix with non-negative entries.  As discussed at the beginning of this section,  the Perron-Frobenius theorem implies that the eigenvector corresponding to the maximum eigenvalue can be chosen to have non-negative entries. Denote this eigenvector by $w$. Using (\ref{eq:2}) and $w\ge 0$, we obtain
\begin{equation}\label{eq:perron}
\<w, \Big( \Phi \circ (D-A) \Big) w \> \ge 
\<w, \Big(   D + \frac{1}{k-1} A \Big) w\> = 
\lambda_{\max}\Big( D + \frac{1}{k-1} A \Big)\,.
\end{equation}
Using the Rayleigh principle, we obtain 
\begin{equation}
\lambda_{\max}\Big( \Phi \circ (D-A) \Big) \ge \<w, \Big( \Phi \circ (D-A) \Big) w\>\,.
\end{equation}
Finally, (\ref{eq:1}) implies $\lambda_{\max}(D-A) \ge \lambda_{\max}\Big( \Phi \circ (D-A) \Big)$. Combining all the inequalities yields the proof.
\end{proof}

Note that it is essential that the eigenvector corresponding to the maximum eigenvalue has non-negative entries. Otherwise, we cannot establish the inequality in (\ref{eq:perron}). Therefore, it does not seem to be possible to generalize these proof techniques to include other eigenvalues besides the maximum eigenvalue as in \cite{elphick15}.

We can now prove that the Kolotilina bound is a lower bound for $\chi_v(G)$. 

\begin{thm}
For any\footnote{We may assume without loss of generality that the adjacency matrix is irreducible, which is equivalent to the graph being connected. The result is true for each connected component.} graph $G$
\begin{equation}
1 + \frac{\mu_1}{\mu_1 - \delta_1 + \lambda_1} \le \chi_v(G).
\end{equation}
\end{thm}

\begin{proof}
The proof is now identical to Lemma~4 in \cite{elphick15}, which was proved by Nikiforov in \cite{nikiforov07}, but with $\chi(G)$ replaced with $\chi_v(G)$. The Kolotilina bound for $\chi_v(G)$ therefore follows immediately when $A$ is the adjacency matrix and $D$ is the diagonal matrix of vertex degrees. The Hoffman bound for $\chi_v(G)$, proved by Galtman and Bilu, follows when $D$ is the zero matrix.
\end{proof}

%%%

\section{Extremal graphs}\label{sec:six}

A graph, $G$, is said to have a \emph{Hoffman coloring} if $\chi(G)$ equals the Hoffman bound. We have investigated graphs with $\chi_v(G) < \chi(G)$ and $\chi_v(G)$ equal to one or more of the bounds proved in this paper. We have found no irregular graph meeting these criteria.

For regular graphs, the Kolotilina and Lima \emph{et al} bounds equal the Hoffman bound, and there are numerous regular graphs for which $\chi_v(G) < \chi(G)$ and $\chi_v(G)$ equals the Hoffman bound. Such graphs can be said to have a \emph{Hoffman vector coloring}. For example the Clebsch graph has $\chi = 5$ and $\chi_v =$ Hoffman bound $= 8/3$; and the Kneser graph $K_{p,k}$ has $\chi = p - 2k + 2$ and $\chi_v = $ Hoffman bound $= p/k$. The orthogonality graph, $\Omega(n)$, has $\chi_v =$ Hoffman bound $= n$ and, for large enough $n$, $\chi$ is exponential in $n$. 

Godsil \emph{et al} \cite{godsil16} proved that any 1-homogeneous graph has $\chi_v =$ Hoffman bound. 1-homogeneous graphs are always regular and include distance regular (and thus strongly regular) and non-bipartite edge transitive graphs; and graphs which are both vertex and edge transitive.

\section{An open question}\label{sec:seven}

As discussed in Section 3, Ando and Lin \cite{ando15} proved a conjecture due to two of the authors \cite{wocjan13} that:

\[
1 + \max\left(   \frac{s^+}{s^-} , \frac{s^-}{s^+} \right) \le \chi(G).
\]

We have been unable to prove that  this bound is also a lower bound for $\chi_v(G)$. We have, however,  tested this question, using that $\chi_v(G) = \vartheta'(\overline{G})$ and SDP, against thousands of named graphs in the Wolfram Mathematica database and found no counter-example. We have also tested 10,000s of circulant graphs and found no counter-example.

Our code for testing this question is available in the GitHub repository \cite{repo}.

\subsection*{Acknowledgements}

This research has been supported in part by National Science Foundation Award 1525943. We would like to thank Ali Mohammadian for helpful comments.

\end{document}